\documentclass[12pt, reqno, twoside, letterpaper]{amsart}

\usepackage{paperstyle}

\title[Non-vanishing of Dirichlet series without Euler products]
      {Non-vanishing of Dirichlet series\break without Euler products}
      
\author[W.\ D.\ Banks]{William D.\ Banks}

\address{Department of Mathematics, 
         University of Missouri, 
         Columbia MO, USA.}

\email{bankswd@missouri.edu}
        
\date{\today}

\begin{document}

\begin{abstract}
We give a new proof that the Riemann zeta function is nonzero
in the half-plane $\{s\in\CC:\sigma>1\}$.  A novel feature of 
this proof is that it makes no use of the Euler product
for $\zeta(s)$.
\end{abstract}

\maketitle


\section{Introduction}
\label{sec:intro}

Let $s=\sigma+it$ be a complex variable. In the half-plane
$$
\sH\defeq\{s\in\CC:\sigma>1\}
$$
the Riemann zeta function can be defined either
as a Dirichlet series
$$
\zeta(s)\defeq\sum_{n\in\NN} n^{-s}
$$
or (equivalently) as an Euler product
$$
\zeta(s)\defeq\prod_{p\text{~prime}}(1-p^{-s})^{-1}.
$$
Since a convergent infinite product of nonzero factors is not zero,
the zeta function does not vanish in $\sH$.
This can also be seen by applying the logarithm to the Euler product:
$$
\log\zeta(s)=\sum_p\sum_{m\in\NN}(mp^s)^{-1}.
$$
Indeed, since the double sum on the right converges absolutely
in $\sH$, it follows that $\zeta(s)\ne 0$ for all $s\in\sH$.
Alternatively, since the M\"obius function $\mu$ is bounded, 
it follows that the series
$$
\zeta(s)^{-1}=\sum_{n=1}^\infty\mu(n)n^{-s}
$$
converges absolutely when $\sigma>1$, so $\zeta(s)$ cannot vanish.
Of course, to prove that the M\"obius function is bounded, one exploits
the multiplicativity of $\mu$, so this argument
also relies (albeit implicitly) on the Euler product for $\zeta(s)$.

It is crucial to our understanding
of the primes to extend the zero-free region for $\zeta(s)$ as far
to the left of $\sigma=1$ as possible.\footnote{At present, the strongest
result in this direction is due to Mossinghoff and Trudjian~\cite{MossTrud};
see also the earlier papers \cite{Ford,JangKwon,Kadiri} and references therein.}
According to Titchmarsh~\cite[\S3.1]{Titch}
this means extending the ``sphere of influence'' of the Euler product:

\medskip

\begin{quote}
{\smaller\sl The problem of the zero-free region appears to be a 
question of extending the sphere of influence of the Euler product beyond
its actual region of convergence; for examples are known of functions which
are extremely like the zeta-function in their representation by Dirichlet
series, functional equation, and so on, but which have no Euler product,
and for which the analogue of the Riemann hypothesis is false.  In fact
the deepest theorems on the distribution of the zeros of $\zeta(s)$ are
obtained in the way suggested.  But the problem of extending the sphere
of influence of [the Euler product] to the left of $\sigma=1$ in any
effective way appears to be of extreme difficulty.}
\end{quote}

\medskip

But let's play the devil's advocate for a moment!
Is it really the case that the non-vanishing of the Riemann zeta function
in $\sH$ (and in wider regions) \emph{fundamentally relies on}
the existence of an Euler product? Our aim in this paper is to provide
some evidence to the contrary.

\section{Statement of results}

For a given arithmetical function $F$ with $F(1)\ne 0$,
let $\widetilde F$ denote the \emph{Dirichlet inverse} of $F$;
this can be defined via the M\"obius relation
$$
\sum_{ab=n}F(a)\widetilde F(b)
=I(n)\defeq\begin{cases}
1&\quad\hbox{if $n=1$};\\
0&\quad\hbox{otherwise}.
\end{cases}
$$
To prove that a Dirichlet series
$D(s)\defeq\sum F(n)n^{-s}$ is nonzero in $\sH$, it
is enough to show that
$D(s)^{-1}=\sum \widetilde F(n)n^{-s}$
converges in $\sH$.  Using partial summation,
this is a consequence of any bound of the form
\begin{equation}
\label{eq:test-tildeF}
\sum_{n\le x}\widetilde F(n)\ll x^{1+o(1)}\qquad(x\to\infty).
\end{equation}
Our proof of the next theorem establishes \eqref{eq:test-tildeF}
whenever $\widetilde F$ is supported
on a set of $\kappa$-free numbers.\footnote{For
a given integer $\kappa\ge 2$, a natural number $n$ is said to be
\emph{$\kappa$-free} if $p^k\nmid n$ for every prime $p$.}

\begin{theorem}
\label{thm:main}
Let $D(s)\defeq\sum_{n\in\NN}F(n)n^{-s}$ be a Dirichlet series
such that $F$ is bounded on $\NN$, $F(1)\ne 0$, and the Dirichlet inverse
$\widetilde F$ is supported on the set of $\kappa$-free
numbers for some $\kappa\ge 2$. Then $D(s)\ne 0$ in $\sH$.
\end{theorem}

This theorem is proved in \S\ref{sec:reciprocal};
it establishes the property of non-vanishing in~$\sH$
for a large class of Dirichlet series, almost all of which do not have
an Euler product (but some do).

For a Dirichlet series $D(s)$ attached to a bounded
completely multiplicative function $F$ (for example, the Riemann zeta function),
Theorem~\ref{thm:main} provides a novel route to showing that $D(s)$
is nonzero in $\sH$.  For such $F$, one can easily show
that $\widetilde F$ is supported on the set of
\emph{squarefree} numbers provided that one has the luxury of using the
Euler product for $D(s)$.
For this reason, it is important to note that our proof of the next theorem
\emph{makes no use of the Euler product for $D(s)$}.
Instead, a combinatorial identity is employed to
show that $\widetilde F$ has the required support.

\begin{theorem}
\label{thm:main2}
Let $F$ be an arithmetical function that is bounded and completely
multiplicative. Then the Dirichlet inverse $\widetilde F$ is supported
on the set of squarefree numbers.
\end{theorem}

This theorem is proved in \S\ref{sec:familyD}.

In particular, Theorems~\ref{thm:main} and~\ref{thm:main2}
together yield the following result
without any use of the Euler product for $\zeta(s)$.

\begin{corollary}
The Riemann zeta function does not vanish in $\sH$.
\end{corollary}

To further illustrate how our results can
be applied, in \S\ref{sec:familyD} we introduce and study a special family of
Dirichlet series $\sD\defeq\{D_z(s):z\in\CC\}$ with the
following properties:
\begin{itemize}
\item[$(i)$] The Riemann zeta function belongs to $\sD$;
\item[$(ii)$] Every series $D_z(s)$ is meromorphic
and \emph{nonzero} in the region $\sH$;
\item[$(iii)$] Only two series in $\sD$ have an Euler product, namely
the Riemann zeta function and the constant function $\ind{\CC}(s)=1$
for all $s\in\CC$.
\end{itemize}
Viewing $\zeta(s)$ in relation to the other members of $\sD$,
the existence of an Euler product seems quite unusual, whereas non-vanishing
in the half-plane $\sH$ is a property enjoyed by every
member of $\sD$.

\section{Preliminaries}

Throughout the paper, we fix an integer parameter $\kappa\ge 2$ and
denote by $\NN_\kappa$ the set of $\kappa$-free numbers.
We denote by $\ind{\NN_\kappa}$ the indicator function of $\NN_\kappa$:
$$
\ind{\NN_\kappa}(n)\defeq\begin{cases}
1&\quad\hbox{if $n\in\NN_\kappa$;}\\
0&\quad\hbox{otherwise}.
\end{cases}
$$

We denote by $\omega(n)$ the number of distinct prime factors of $n$
and by $\Omega(n)$ the number of prime factors of $n$, counted with
multiplicity.

For any integer $k\ge 2$ we denote by $\log_kx$ the $k$-th iterate
of the function $x\mapsto \max\{\log x,1\}$.  In particular,
$\log_2x=\log\log x$ and $\log_3x=\log\log\log x$
when $x$ is sufficiently large.

We use the equivalent notations $f(x)=O(g(x))$ and
$f(x)\ll g(x)$ to mean that
the inequality $|f(x)|\le c\,g(x)$ holds with some constant $c$.
Throughout the paper, any implied constants in the symbols $O$ and $\ll$
may depend (where obvious) on the parameters $\kappa,\eps$ but
are absolute otherwise.

Two classical results of Hardy and Ramanujan~\cite[Lemmas B and C]{HardyRaman1}
assert the existence of constants $c_1,c_2>0$ such that the
inequalities
\begin{equation}
\label{eq:HarRam1}
\big|\{n\le x:\omega(n)=\ell\}\big|\le\frac{c_1x}{\log x}
\frac{(\log_2x+c_2)^{\ell-1}}{(\ell-1)!}
\end{equation}
and
\begin{equation}
\label{eq:HarRam2}
\big|\{n\le x:\Omega(n)=\ell\}\big|\le\frac{c_1x}{\log x}
\sum_{j=0}^{\ell-1}\big(\tfrac{9}{10}\big)^{\ell-1-j}
\frac{(\log_2x+c_2)^j}{j!}
\end{equation}
hold for all real $x\ge 2$.  In the next lemma,
we study the counting function
\begin{equation}
\label{eq:Nklx-defn}
N_{\kappa,\ell}(x)\defeq
\big|\{n\le x:n\in\NN_\kappa\text{~and~}\Omega(n)=\ell\}\big|.
\end{equation}
Although this function might seem closely related to
that on the left side of \eqref{eq:HarRam2}, we prove that it
satisfies a bound nearly as strong as \eqref{eq:HarRam1}.

\begin{lemma}
\label{lem:Ryan}
There are absolute constants $C_1,C_2>0$ with the following property.
For any integers $\kappa\ge 2$ and $\ell\ge 1$, the counting function
defined by \eqref{eq:Nklx-defn} satisfies the upper bound
\begin{equation}
\label{eq:HarRam-free}
N_{\kappa,\ell}(x)\le
\frac{C_1x}{\log x}\frac{((\kappa-1)\log_2x+(\kappa-1)C_2)^{\ell-1}}{(\ell-1)!}
\qquad(x\ge 2).
\end{equation}
\end{lemma}

\begin{proof}
Our proof is an adaptation of arguments from \cite{HardyRaman1}.

When $\ell\le\kappa$, the condition $\Omega(n)=\ell$
implies that $n\in\NN_\kappa$.  Using \eqref{eq:HarRam2} it follows that
$$
N_{\kappa,\ell}(x)
=\big|\{n\le x:\Omega(n)=\ell\}\big|
\le\frac{ec_1x(\log_2x+c_2)^{\ell-1}}{\log x}\qquad(x\ge 2),
$$
hence \eqref{eq:HarRam-free} holds
for any choice of $C_1\ge ec_1$ and $C_2\ge c_2$.

From now on, we assume that $\ell>\kappa$. To simplify the notation slightly,
we put $\kappa_1\defeq\kappa-1$.

Let $p(1)\defeq 2<p(2)\defeq 3<p(3)\defeq 5<\cdots$ be the sequence of all 
primes, and put
$$
\tilde p(j)\defeq p(\rf{j/\kappa_1})\qquad(j\in\NN),
$$
where for any $t>0$, $\rf{t}$ is the least integer that is $\ge t$.
In other words, $\big(\tilde p(j)\big)_{j\in\NN}$ is the sequence
$$
\underbrace{\,2,\ldots,2\,}_{\kappa_1\text{~copies}}\,,
\underbrace{\,3,\ldots,3\,}_{\kappa_1\text{~copies}}\,,
\underbrace{\,5,\ldots,5\,}_{\kappa_1\text{~copies}}\,,\ldots
$$
in which the primes appear in increasing order, each being
repeated $\kappa_1$ times.

Let $n\in\NN_\kappa$, $n\ge 2$, and suppose that $\Omega(n)=\ell$.
Among all of the ordered $\ell$-tuples $(j_1,\ldots,j_\ell)$ having
$j_1<\cdots<j_\ell$ and for which
\begin{equation}
\label{eq:factorization}
n=\tilde p(j_1)\cdots\tilde p(j_\ell),
\end{equation}
let $\Psi(n)$ be the unique $\ell$-tuple $(j_1,\ldots,j_\ell)$ 
that minimizes the sum $j_1+\cdots+j_\ell$.  For any such $n$
we also put
$$
J(n)\defeq j_\ell,
$$
and we set $J(1)\defeq 0$.  For example, if $\kappa=5$,
then $4400\in\NN_\kappa$ and we have
$$
\Psi(4400)=(1, 2, 3, 4, 9, 10, 17)\mand J(4400)=17.
$$

Let $\cS$ be the set of ordered pairs $(j,m)$ such that $j\le J(m)$,
$m\in\NN_\kappa$, and $\tilde p(j)m\le x$.  The condition
$j\le J(m)$ implies that the prime $\tilde p(j)$ does not exceed
the largest prime factor of $m$, and thus $\tilde p(j)\le m\le x/\tilde p(j)$; consequently,
\begin{equation}
\label{eq:dough1}
|\cS|\le\sum_{j\,:\,\tilde p(j)^2\le x}N_{\kappa,\ell-1}(x/\tilde p(j)).
\end{equation}
On the other hand, suppose that $n\le x$, $n\in\NN_\kappa$
and $\Omega(n)=\ell$.  Factoring $n$ as in \eqref{eq:factorization}
with $(j_1,\ldots,j_\ell)=\Psi(n)$, one verifies that the pair
$(j_i,n/\tilde p(j_i))$ lies in $\cS$ for each $i=1,\ldots,\ell-1$
Hence, $n$ can be expressed in $\ell-1$ different ways
as the product of the entries of an ordered pair in $\cS$,
which implies that
\begin{equation}
\label{eq:dough2}
(\ell-1)N_{\kappa,\ell}(x)\le |\cS|.
\end{equation}
Combining \eqref{eq:dough1} and \eqref{eq:dough2}, and using induction,
we have
\begin{align*}
N_{\kappa,\ell}(x)
&\le\frac{1}{\ell-1}\sum_{j\,:\,\tilde p(j)^2\le x}N_{\kappa,\ell-1}(x/\tilde p(j))\\
&\le\frac{1}{\ell-1}\sum_{j\,:\,\tilde p(j)^2\le x}
\frac{C_1(x/\tilde p(j))}{\log(x/\tilde p(j))}
\frac{(\kappa_1\log_2(x/\tilde p(j))+\kappa_1C_2)^{\ell-2}}{(\ell-2)!}\\
&\le\frac{C_1x(\kappa_1\log_2x+\kappa_1C_2)^{\ell-2}}{(\ell-1)!}
\sum_{j\,:\,\tilde p(j)^2\le x}\frac{1}{\tilde p(j)\log(x/\tilde p(j))}.
\end{align*}
As each prime is repeated precisely $\kappa_1$ times in
$\big(\tilde p(j)\big)_{j\in\NN}$ we have
$$
\sum_{j\,:\,\tilde p(j)^2\le x}\frac{1}{\tilde p(j)\log(x/\tilde p(j))}
=\kappa_1\sum_{p\,:\,p^2\le x}\frac{1}{p\log(x/p)}.
$$
The proof is completed using the fact that
$$
\sum_{p\,:\,p^2<x}\frac{1}{p\log(x/p)}<\frac{\log_2x+C_2}{\log x}\qquad(x\ge 2)
$$
holds for a sufficiently large choice of $C_2$;
see the proof of \cite[Lemma~C]{HardyRaman1}.
\end{proof}

\section{Factorisatio numerorum}
\label{sec:FactNumer}

For any $n\ge 2$, let
$f(n)$ denote the number of representations of $n$
as a product of integers exceeding one, where two representations are considered
equal only if they contain the same factors in the same order.  For technical
reasons, we also set $f(1)\defeq 1$.  

One can define $f(n)$ as follows.
For every positive integer $k$, let $f_k(n)$ denote the
number of ordered $k$-tuples $(n_1,\ldots,n_k)$ such that each $n_j\ge 2$ and
the product $n_1\cdots n_k$ equals $n$.  Then
\begin{equation}
\label{eq:dinftysum}
f(n)\defeq I(n)+\sum_{k\ge 1}f_k(n)\qquad(n\in\NN),
\end{equation}
where
$$
I(n)\defeq\begin{cases}
1&\quad\hbox{if $n=1$};\\
0&\quad\hbox{otherwise}.
\end{cases}
$$
Note that the sum in \eqref{eq:dinftysum}
is finite since $f_k(n)=0$ when $k>\Omega(n)$.

In one of the earliest papers about the function $f(n)$,
Kalm\'ar~\cite{Kalmar} establishes the asymptotic formula
\begin{equation}
\label{eq:Kalmar}
\sum_{n\le x}f(n)\sim -\frac{x^\beta}{\beta\zeta'(\beta)}\qquad(x\to\infty),
\end{equation}
where $\beta=1.728647\cdots$ is the unique positive root of $\zeta(\beta)=2$.
In particular, this implies that
\begin{equation}
\label{eq:f(n)bound}
f(n)\ll n^\beta\qquad(n\in\NN).
\end{equation}
The bound \eqref{eq:f(n)bound} is essentially optimal since Hille~\cite{Hille}
has shown that for every $\eps>0$ the lower bound $f(n)\gg n^{\beta-\eps}$ holds for
infinitely many $n$; see also Erd\H os~\cite{Erdos}.

The next proposition is fundamental in the sequel as it leads to
a significant strengthening of \eqref{eq:f(n)bound} for $\kappa$-free
numbers $n$.

\begin{proposition}
\label{prop:powerade}
For any integer $n\ge 2$ we have
$$
f(n)\le\exp(\ell\log\ell+O(\ell\log_2\ell\log_3\ell))
\qquad\text{with}\quad\ell\defeq\Omega(n).
$$
\end{proposition}

\begin{proof}
Let $\cT(n)$ be the set of ordered tuples $(n_1,\ldots,n_r)$ of any
length $r$ such that every $n_j\ge 2$
and $n_1\cdots n_r=n$.  Thus, $|\cT(n)|=f(n)$.

Let $\cP(\ell)$ be the set
of ordered partitions $\lambda$ of $\ell$; these are ordered tuples
$\lambda=(\lambda_1,\ldots,\lambda_r)$ of any length $r$
such that $1\le\lambda_1\le\cdots\le\lambda_r$ and
$\lambda_1+\cdots+\lambda_r=\ell$.

We begin by constructing a map $\Phi:\cT(n)\to\cP(\ell)$ as follows.  For any given
$\eta=(n_1,\ldots,n_r)$ in $\cT(n)$, let $\Phi_\Omega(\eta)$ denote the
tuple $(\Omega(n_1),\ldots,\Omega(n_r))$, and set
$$
\cU(n)\defeq\{\Phi_\Omega(\eta):\eta\in\cT(n)\};
$$
thus, $\Phi_\Omega:\cT(n)\to\cU(n)$.  Next, for any $w=(w_1,\ldots,w_r)$
in $\cU(n)$, let $\Phi_\sigma(w)$ be the tuple
$(\lambda_1,\ldots,\lambda_r)$ that is obtained by rearranging the entries of $w$
into nondecreasing order; then $\Phi_\sigma:\cU(n)\to\cP(\ell)$.
The map $\Phi:\cT(n)\to\cP(\ell)$ is defined to be the composition
$\Phi_\sigma\circ\Phi_\Omega$.

Next, for any $\lambda\in\cP(\ell)$ let $d_\lambda(n)$ 
be the cardinality of the set $\Phi^{-1}(\{\lambda\})$
of preimages of $\lambda$ in $\cT(n)$.\footnote{A more descriptive but less precise
definition is the following. If $\lambda=(\lambda_1,\ldots,\lambda_r)\in\cP(\ell)$,
then $d_\lambda(n)$ is the number of $r$-tuples $(n_1,\ldots,n_r)$
for which the product $n_1\cdots n_r$ equals $n$ and such that,
after a suitable permutation of the indices, one has
$\Omega(n_j)=\lambda_j$ for each $j$
(that is, the multisets $\{\Omega(n_j)\}$ and $\{\lambda_j\}$ are the same).}
Since any product counted by $f(n)$ gives rise to a unique
partition $\lambda$ via the map $\Phi$, we have
$$
f(n)\defeq I(n)+\sum_{\lambda\in\cP(\ell)} d_\lambda(n).
$$
In view of the celebrated estimate of
Hardy and Ramanujan~\cite{HardyRaman2}
$$
|\cP(\ell)|\sim(4\ell\sqrt{3})^{-1}\exp\big(\pi\sqrt{2\ell/3}\,\big)
\qquad(\ell\to\infty),
$$
to prove the proposition it suffices to show that the individual bound
\begin{equation}
\label{eq:flam(n)}
d_\lambda(n)\le\exp(\ell\log\ell+O(\ell\log_2\ell\log_3\ell))
\end{equation}
holds for every $\lambda\in\cP(\ell)$.

To this end, let $\lambda=(\lambda_1,\ldots,\lambda_r)$
be a fixed element of $\cP(\ell)$. For any natural number $k$,
let $m_k$ be the multiplicity with which $k$ occurs
in the partition $\lambda$, i.e.,
$$
m_k\defeq\big|\{j:\lambda_j=k\}\big|\qquad(k\in\NN).
$$
Note that $\ell=\sum_k km_k$.
Setting $m\defeq\sum_k m_k$, a simple combinatorial argument
shows that
\begin{equation}
\label{eq:wonderwoman}
d_\lambda(n)\le\frac{\ell!}{\prod_k(k!)^{m_k}}\cdot\frac{m!}{\prod_k m_k!}
\end{equation}
(roughly speaking, the second factor is the cardinality of the set
$\Phi_\sigma^{-1}(\{\lambda\})$ of preimages of $\lambda$ in $\cU(n)$,
whereas for any such preimage $w$ the first factor bounds the cardinality
of the set $\Phi_\Omega^{-1}(\{w\})$ of preimages of $w$ in $\cT(n)$).
We remark that \eqref{eq:wonderwoman}
holds with equality whenever $n$ is squarefree. Since
$\ell!\le\ell^\ell$, to establish \eqref{eq:flam(n)} it is enough to show that
\begin{equation}
\label{eq:robotics}
\log\bigg(\frac{m!}{\prod_k(k!)^{m_k}\prod_k m_k!}\bigg)\ll\ell\log_2\ell\log_3\ell.
\end{equation}
Since $m\le\ell$ and $\log j!=j\log j+O(j)$ for all positive
integers $j$, the left side of \eqref{eq:robotics} is
\begin{align*}
&\le m\log\ell-\sum_k m_k(k\log k+O(k))-\sum_k(m_k\log m_k+O(m_k))\\
&=\sum_{k\,:\,m_k\ne 0} m_k\log\Big(\frac{\ell}{k^k m_k}\Big)+O(\ell)
=S_1+S_2+O(\ell),\quad\text{(say)}
\end{align*}
where
$$
S_1\defeq\sum_{\substack{k\,:\,m_k\ne 0\\m_k>\ell/g(\ell)}}
m_k\log\Big(\frac{\ell}{k^k m_k}\Big)
\mand
S_2\defeq\sum_{\substack{k\,:\,m_k\ne 0\\m_k\le\ell/g(\ell)}}
m_k\log\Big(\frac{\ell}{k^k m_k}\Big)
$$
and
$$
g(\ell)\defeq\frac{(\log\ell)^2}{(\log_2\ell)^2\log_3\ell}.
$$
For each $k$ in the sum $S_1$, we have
$m_k>\ell/g(\ell)$ and $km_k\le\sum_{j\le\ell} jm_j=\ell$;
therefore,
\begin{align*}
S_1\le\sum_{k}
\frac{\ell}{k}\log\Big(\frac{g(\ell)}{k^k}\Big)
&\le\ell\log g(\ell)\sum_{k\,:\,k^k\le g(\ell)}\frac{1}{k}\\
&\ll\ell\log g(\ell)\log_2g(\ell)\ll\ell\log_2\ell\log_3\ell.
\end{align*}
For each $k$ in the sum $S_2$, we have $1\le m_k\le\ell/g(\ell)$; thus,
\begin{align*}
S_2\le\frac{\ell}{g(\ell)}\sum_k
\log\Big(\frac{\ell}{k^k}\Big)
&\le\frac{\ell\log\ell}{g(\ell)}\sum_{k\,:\,k^k\le\ell}1
\ll\frac{\ell(\log\ell)^2}{g(\ell)\log_2\ell}\ll\ell\log_2\ell\log_3\ell.
\end{align*}
Combining the above bounds on $S_1$ and $S_2$, we derive \eqref{eq:robotics},
and in turn \eqref{eq:flam(n)}, finishing the proof.
\end{proof}

The following corollary is crucial in the next section.

\begin{corollary}
\label{cor:coffeeshop}
For any constant $C>0$ we have
\begin{equation}
\label{eq:okko}
\bigg|\sum_{n\le x}C^{\Omega(n)}f(n)\ind{\NN_\kappa}(n)\bigg|\le x^{1+o(1)}
\qquad(x\to\infty),
\end{equation}
where the function implied by $o(1)$ depends only on $C$ and $\kappa$.
\end{corollary}

\begin{proof}
Let $Q$ denote the quantity on the left side of \eqref{eq:okko}.

For any $n\in\NN_\kappa$ we have $\Omega(n)\le\kappa\,\omega(n)$.
Also, $\omega(n)\le 2(\log x)/\log_2x$ for all $n\le x$
once $x$ is sufficiently large.  Hence, defining
$B_\kappa(x)\defeq 2\kappa(\log x)/\log_2x$ it follows from
Proposition~\ref{prop:powerade} that
$$
Q\le\sum_{\ell\le B_\kappa(x)}
C^\ell\exp(\ell\log\ell+O(\ell\log_2\ell\log_3\ell))
\cdot N_{\kappa,\ell}(x),
$$
where $N_{\kappa,\ell}(x)$ is the counting function given by \eqref{eq:Nklx-defn}.
By Lemma~\ref{lem:Ryan} we have
\begin{align*}
Q&\le\sum_{\ell\le B_\kappa(x)}
C^\ell\exp(\ell\log\ell+O(\ell\log_2\ell\log_3\ell))
\cdot \frac{C_1x}{\log x}\frac{((\kappa-1)\log_2x+(\kappa-1)C_2)^{\ell-1}}{(\ell-1)!}\\
&\le x^{1+o(1)}\sum_{\ell\le B_\kappa(x)}
\exp(\ell\log\ell)\cdot
\frac{((\kappa-1)\log_2x+(\kappa-1)C_2)^{\ell-1}}{(\ell-1)!}\qquad(x\to\infty).
\end{align*}
Using the estimates
$$
(\ell-1)!=\exp(\ell\log\ell+O(\ell))=x^{o(1)}\exp(\ell\log\ell)
$$
and
$$
((\kappa-1)\log_2x+(\kappa-1)C_2)^{\ell-1}=\exp(O(\ell\log_3x))=x^{o(1)},
$$
which hold uniformly for all $\ell\le B_\kappa(x)$,
the result follows.
\end{proof}

\section{Reciprocal of a Dirichlet series}
\label{sec:reciprocal}

\begin{theorem}
\label{thm:lockdown}
Suppose that $F$ is bounded on $\NN$, and $F(1)\ne 0$.
Then the Dirichlet series $\sum_{n\in\NN}\widetilde F(n)\ind{\NN_\kappa}(n)n^{-s}$
converges absolutely in $\sH$.
\end{theorem}

\begin{proof}
Without loss of generality,
we can assume that $F(1)=\widetilde F(1)=1$.
Let $C\ge 1$ be a number such that
\begin{equation}
\label{eq:est1}
|F(n)|\le C\qquad(n\in\NN).
\end{equation}

For every positive integer $k$, let $\cT_k(n)$ be the
set of ordered $k$-tuples $(n_1,\ldots,n_k)$ such that every $n_j\ge 2$ and
$n_1\cdots n_k=n$.  Then $|\cT_k(n)|=f_k(n)$ in the notation of
\S\ref{sec:FactNumer}.  We denote
\begin{equation}
\label{eq:fkFn}
f_k(F;n)\defeq\sum_{(n_1,\ldots,n_k)\in\cT_k(n)}F(n_1)\cdots F(n_k).
\end{equation}
Using \eqref{eq:est1} we derive that
\begin{equation}
\label{eq:est2}
|f_k(F;n)|\le C^kf_k(n).
\end{equation}
Since $f_k(F;n)=0$ for all $k>\Omega(n)$, and the inequality
$\Omega(n)\le(\log n)/\log 2$ holds for all $n$, it follows that
$$
|f_k(F;n)|\le n^Bf_k(n)\qquad\text{with}\quad
B\defeq \max\{0,(\log C)/\log 2\}.
$$
Summing over $k$ and using \eqref{eq:f(n)bound}, we see that
\begin{equation}
\label{eq:est3}
\sum_{k\ge 1}|f_k(F;n)|\ll n^{B+\beta}\qquad(n\in\NN).
\end{equation}

Next, put
$$
D(s)\defeq\sum_{n=1}^\infty F(n)n^{-s}=1+Z(F;s)
\qquad\text{with}\quad
Z(F;s)\defeq\sum_{n\ge 2}F(n)n^{-s}.
$$
For every positive integer $k$ we have
$$
Z(F;s)^k=\sum_{n\ge 2} f_k(F;n)n^{-s}.
$$
In view of \eqref{eq:est3} the identity
\begin{equation}
\label{eq:formal}
\frac{1}{1+Z(F;s)}
=1+\sum_{k\ge 1}(-1)^k Z(F;s)^k
=1+\sum_{n\ge 2}\sum_{k\ge 1}(-1)^k f_k(F;n)n^{-s}
\end{equation}
holds throughout the half-plane $\{\sigma>B+\beta\}$
since all sums converge absolutely in that region.
Noting that the left side
of \eqref{eq:formal} is $D(s)^{-1}=\sum_{n\in\NN}\widetilde F(n)n^{-s}$, we conclude that
\begin{equation}
\label{eq:tildeFreln}
\widetilde F(n)=I(n)+\sum_{k\ge 1}(-1)^k f_k(F;n)\qquad(n\in\NN).
\end{equation}

To prove the theorem, we need to show that the Dirichlet series
$$
\sum_{n\in\NN}\widetilde F(n)\ind{\NN_\kappa}(n)n^{-s}
=1+\sum_{n\ge 2}\sum_{k\ge 1}(-1)^k f_k(F;n)\ind{\NN_\kappa}(n)n^{-s}
$$
converges absolutely in $\sH$.  For any natural number $n$ 
it is clear that $f_k(F;n)=0$ whenever $k>\Omega(n)$, hence using
\eqref{eq:est2} we see that
\begin{align*}
\bigg|\sum_{k\ge 1}(-1)^k f_k(F;n)\ind{\NN_\kappa}(n)\bigg|
\le I(n)+\sum_{k\le \Omega(n)}C^kf_k(n)\ind{\NN_\kappa}(n)
\le C^{\Omega(n)}f(n)\ind{\NN_\kappa}(n).
\end{align*}
Consequently, it suffices to show that the sum
\begin{equation}
\label{eq:somesum}
\sum_{n\in\NN}C^{\Omega(n)}f(n)\ind{\NN_\kappa}(n)n^{-\sigma}
\end{equation}
converges for $\sigma>1$.  However, since the summatory function
$$
S(x)\defeq\sum_{n\le x}C^{\Omega(n)}f(n)\ind{\NN_\kappa}(n)
$$
satisfies the bound $S(x)\ll x^{1+\eps}$ by Corollary~\ref{cor:coffeeshop},
the convergence of \eqref{eq:somesum} for $\sigma>1$ follows
by partial summation.
\end{proof}

\begin{proof}[Proof of Theorem~\ref{thm:main}]
Since the Dirichlet inverse $\widetilde F$ has its support in 
$\NN_\kappa$ for some $\kappa\ge 2$,
we have $\widetilde F(n)=\widetilde F(n)\ind{\NN_\kappa}(n)$
for all $n$.  By Theorem~\ref{thm:lockdown},
$$
\sum_{n\in\NN}\widetilde F(n)\ind{\NN_\kappa}(n)n^{-s}
=\sum_{n\in\NN}\widetilde F(n)n^{-s}
=D(s)^{-1}
$$
converges absolutely in $\sH$, and the result follows.
\end{proof}

\section{The family $\sD$}
\label{sec:familyD}

For any fixed $z\in\CC$, and let $F_z$ be the arithmetical function defined by
$$
F_z(n)\defeq\begin{cases}
1&\quad\hbox{if $n=1$};\\
-z&\quad\hbox{otherwise}.
\end{cases}
$$
Then
$$
\sum_{n\in\NN}F_z(n)n^{-s}=1-z(\zeta(s)-1)\qquad(s\in\sH).
$$
Taking $F\defeq F_z$ in \eqref{eq:est3} and \eqref{eq:tildeFreln}
we have the bound $\widetilde F_z(n)\ll n^{B+\beta}$, where
$B\defeq \max\{0,(\log|z|)/\log 2\}$ and $\beta=1.728647\cdots$ as before.
This implies that the formal identity
\begin{equation}
\label{eq:AGT}
\sum_{n\in\NN}\widetilde F_z(n)n^{-s}=\frac{1}{1-z(\zeta(s)-1)}
\end{equation}
holds rigorously when $\sigma>B+\beta+1$.  Moreover, it is clear that
the Dirichlet series can be analytically
continued to the region $\{\sigma>\beta_z\}$, where
$\beta_z$ is the unique positive root of $\zeta(\beta_z)=1+|z|^{-1}$
if $z\ne 0$, and $\beta_0\defeq-\infty$.  It is worth mentioning that
for any fixed $z\ne 0$ or $-1$, the function on the right side of \eqref{eq:AGT}
has infinitely many poles in $\sH$ since the equation $\zeta(s)=1+z^{-1}$
has infinitely many solutions in any strip $\{1<\sigma<1+\eps\}$;
see, e.g., Titchmarsh~\cite[Theorem~11.6\,(C)]{Titch}.

Next, we introduce two Dirichlet series given by
$$
D_z^\dagger(s)\defeq\sum_{n\in\NN}\widetilde F_z(n)\mu(n)^2n^{-s}
$$
and
$$
D_z(s)\defeq D_z^\dagger(s)^{-1}=\sum_{n\in\NN}G_z(n)n^{-s},
$$
where $G_z$ is the Dirichlet inverse of $\widetilde F_z\cdot\mu^2$.
According to Theorem~\ref{thm:lockdown}, $D_z^\dagger(s)$
converges absolutely in $\sH$, hence it is analytic in that region.
This implies that $D_z(s)$ has a meromorphic
extension to $\sH$, and $D_z(s)\ne 0$ in $\sH$;
thus, we have verified property $(ii)$ of \S\ref{sec:intro}.

From the above definitions, one sees that $F_{-1}(n)=\ind{\NN}(n)$,
thus $D_{-1}(s)=\zeta(s)$ (establishing property $(i)$ of \S\ref{sec:intro}), and
$F_0(n)=I(n)$, so that $D_0(s)=\ind{\CC}(s)$.

To establish property $(iii)$ of \S\ref{sec:intro},
observe that the M\"obius relations
$$
\sum_{ab=n}\widetilde F_z(a)F_z(b)=I(n)
\mand\sum_{ab=n}\widetilde F_z(a)\mu(a)^2G_z(b)=I(n)
$$
immediately imply that
$
G_z(p)=G_z(q)=G_z(pq)=-z
$
for any two different primes $p$ and $q$.\footnote{A more
elaborate argument shows that $F_z(n)=G_z(n)=-z$
for all squarefree numbers $n$.}
If $D_z(s)$ has an Euler product, then
$G_z$ is multiplicative, and therefore
$$
(-z)^2=G_z(p)G_z(q)=G_z(pq)=-z,
$$
which is only possible for $z=0$ or $-1$.

\begin{lemma}
\label{lem:chirpingbird}
Let $z$ be a complex number, $n$ a natural number,
and $p$ a prime number not dividing $n$.
For any integer $\alpha\ge 1$ we have
$$
\widetilde F_z(p^\alpha n)
=(z+1)^{\alpha-1}
\sum_{\ell\ge 1}z^\ell \big(z+\ell\alpha^{-1}(z+1)\big)
\binom{\alpha+\ell-1}{\ell}
f_\ell(n).
$$
\end{lemma}

\begin{proof}
Using \eqref{eq:fkFn}, \eqref{eq:tildeFreln} and the definition
of $F_z$, it follows that
\begin{equation}
\label{eq:dancearts}
\widetilde F_z(p^\alpha n)=\sum_{k\ge 1}z^k f_k(p^\alpha n).
\end{equation}
The quantity $f_k(p^\alpha n)$ is the number of ordered $k$-tuples
$(m_1,\ldots,m_k)$ such that $m_j=p^{\alpha_j}n_j$ for each $j$,
with $\alpha_j\ge 1$ or $n_j\ge 2$,
$\alpha_1+\cdots+\alpha_k=\alpha$, and $n_1\cdots n_k=n$.
To construct such a $k$-tuple, first choose an integer $\ell$
in the range $1\le\ell\le k$
and an ordered $\ell$-tuple $(\hat n_1,\ldots,\hat n_\ell)$ with
each $\hat n_j\ge 2$ and $\hat n_1\cdots\hat n_\ell=n$; for any choice of $\ell$
there are precisely $f_\ell(n)$ such $\ell$-tuples.  Next,
maintaining the ordering of the integers $\hat n_j$ in
$(\hat n_1,\ldots,\hat n_\ell)$,
we construct $(n_1,\ldots,n_k)$ by
inserting $k-\ell$ extra entries, each equal to one (thus,
every $n_i$ in the resulting \text{$k$-tuple}
is one of the numbers $\hat n_j$,
or else $n_i=1$); there are $\binom{k}{k-\ell}=\binom{k}{\ell}$ ways to
insert these extra ones to form $(n_1,\ldots,n_k)$.
To guarantee that every $m_j=p^{\alpha_j}n_j\ge 2$ 
in the final $k$-tuple $(m_1,\ldots,m_k)$, so it is
counted in the computation of $f_k(p^\alpha n)$, we have to replace each
entry $n_i=1$ in $(n_1,\ldots,n_k)$ with a copy
of the prime $p$.  As
there are only $\alpha$ copies of $p$ available,
it must be the case that $\alpha\ge k-\ell$ else this choice
of $\ell$ is unacceptable. The
remaining $\alpha-k+\ell$ copies of the prime $p$ can be distributed arbitrarily.
As the number of ways that one can distribute $\alpha-k+\ell$ objects
into $k$ boxes is $\binom{\alpha+\ell-1}{k-1}$, putting everything
together we have
$$
f_k(p^\alpha n)=\sum_{\ell=\max\{1,k-\alpha\}}^k\binom{k}{\ell}
\binom{\alpha+\ell-1}{k-1}f_\ell(n).
$$
Combining this result with \eqref{eq:dancearts}, we derive that
$$
\widetilde F_z(p^\alpha n)
=\sum_{\ell\ge 1}f_\ell(n)
\sum_{k=\ell}^{\alpha+\ell}z^k
\binom{k}{\ell}
\binom{\alpha+\ell-1}{k-1}
$$
Making the change of variables $k\mapsto k+\ell$ in the inner
summation, it follows that
$$
\widetilde F_z(p^\alpha n)
=\sum_{\ell\ge 1}z^\ell f_\ell(n)B(z,\alpha,\ell),
$$
where
$$
B(z,\alpha,\ell)\defeq\sum_{k=0}^\alpha z^k\binom{k+\ell}{\ell}
\binom{\alpha+\ell-1}{k+\ell-1}.
$$
In view of the combinatorial identity
$$
\binom{k+\ell}{\ell}
\binom{\alpha+\ell-1}{k+\ell-1}
=\binom{\alpha+\ell-1}{\ell}
\bigg(\binom{\alpha-1}{k-1}
+\frac{\ell}{\alpha}\binom{\alpha}{k}\bigg)
$$
(where $\binom{\alpha-1}{k-1}=0$ when $k=0$),
it follows that
\begin{align*}
B(z,\alpha,\ell)&=\binom{\alpha+\ell-1}{\ell}
\sum_{k=0}^\alpha z^k
\bigg(\binom{\alpha-1}{k-1}
+\frac{\ell}{\alpha}\binom{\alpha}{k}\bigg)\\
&=\binom{\alpha+\ell-1}{\ell}
\big(z(z+1)^{\alpha-1}+\ell\alpha^{-1}(z+1)^\alpha\big),
\end{align*}
and we obtain the stated result.
\end{proof}

\begin{proof}[Proof of Theorem~\ref{thm:main2}]
Since $F$ is completely multiplicative, from \eqref{eq:fkFn}
it follows that
$$
f_k(F;n)=F(n)f_k(n)\qquad(k,n\in\NN).
$$
With two applications of \eqref{eq:tildeFreln} we deduce that
\begin{equation}
\label{eq:lunchtime}
\widetilde F(n)=I(n)+F(n)\sum_{k\ge 1}(-1)^k f_k(n)
=F(n)\widetilde F_{-1}(n)\qquad(n\in\NN).
\end{equation}
Applying Lemma~\ref{lem:chirpingbird} with $z=-1$, we see that
$\widetilde F_{-1}(p^\alpha n)=0$ for any $n\in\NN$, any prime $p$ not
dividing $n$, and all $\alpha\ge 2$.  This implies that $\widetilde F_{-1}$
is supported on the set of squarefree numbers,\footnote{As we have already seen,
$\widetilde F_{-1}$ is the M\"obius function $\mu$.}
and \eqref{eq:lunchtime} shows that the same is true of $F$.
\end{proof}

\section{Remarks}
\label{sec:Sarnak}

Theorems~\ref{thm:main} and \ref{thm:lockdown} can be extended
to cover all functions satisfying the polynomial growth 
condition $F(n)\ll n^A$ provided that one is willing to replace $\sH$
with the half-plane $\{s\in\CC:\sigma>A+1\}$ in those theorems.
It would be interesting to see whether the ideas of this paper can be
developed to produce zero-free regions for $\zeta(s)$ and other Dirichlet
series inside the critical strip.

Sarnak~\cite{Sarnak} has recently considered a general pseudo-randomness principle
related to a famous conjecture of Chowla\cite{Chowla}.  Roughly speaking, the
principle asserts that the M\"obius function $\mu(n)$ does not correlate with
any function $\xi(n)$ of low complexity.  In other words,
\begin{equation}
\label{eq:Sarnak}
\sum_{n\le x}\mu(n)\xi(n)=o\bigg(\sum_{n\le x}|\xi(n)|\bigg)\qquad(x\to\infty).
\end{equation}
Combining Kalm\'ar's result~\eqref{eq:Kalmar} with Corollary~\ref{cor:coffeeshop},
we see that \eqref{eq:Sarnak} is verified for the function $\xi(n)\defeq f(n)$.  However,
this is not due to the randomness of $\mu(n)$ but instead to the fact $f(n)$
takes smaller values on squarefree numbers than it does on natural numbers in general.
It would be interesting see whether \eqref{eq:Sarnak} holds for 
$\xi(n)\defeq f(n)\mu(n)^2$.

Let $f_{\it even}(n)$ [resp.\ $f_{\it odd}(n)$] denote
 the number of representations of $n$
as a product of an \emph{even} [resp.\ \emph{odd}] number of integers
exceeding one, where two representations are considered
equal only if they contain the same factors in the same order.
In other words,
$$
f_{\it even}(n)\defeq I(n)+\sum_{\substack{k\ge 1\\k\text{~even}}}f_k(n)
\mand
f_{\it odd}(n)\defeq \sum_{\substack{k\ge 1\\k\text{~odd}}}f_k(n).
$$
Clearly,
$f(n)=f_{\it even}(n)+f_{\it odd}(n)$,
but it is less obvious that
\begin{equation}
\label{eq:mu-relation}
\mu(n)=f_{\it even}(n)-f_{\it odd}(n)\qquad(n\in\NN).
\end{equation}
Indeed, taking $F\defeq F_{-1}=\ind{\NN}$ we have $f_k(F;n)=f_k(n)$
for all $n$ by \eqref{eq:fkFn}, and then \eqref{eq:mu-relation}
follows immediately from \eqref{eq:tildeFreln}.

\section*{Acknowledgements}
The author thanks Andrew Granville and Igor Shparlinski for their insightful
comments on the original draft, which led to improvements in the
results and the exposition.

\end{document}